\documentclass[12pt]{amsart}
\usepackage{amsmath, amssymb}
\usepackage{enumerate} 
\makeatletter
\@namedef{subjclassname@2020}{\textup{2020} Mathematics Subject Classification}
\makeatother
\newtheorem{theorem}{\bf Theorem}[section]
\newtheorem{lemma}[theorem]{\bf Lemma}
\newtheorem{proposition}[theorem]{\bf Proposition}
\newtheorem{corollary}[theorem]{\bf Corollary}

\newtheorem{example}[theorem]{\sc Example}

\newcommand{\N}{\mathbb{N}}
\newcommand{\al }{\alpha }
\newcommand{\pr }{\mathrm{Pr} }

\newcommand{\ep}{\epsilon}

\newcommand{\LL }{L }
\newcommand{\R }{Q }

\newcommand{\elp}{p}

\DeclareMathOperator{\aut}{Aut}
\DeclareMathOperator{\out}{Out}
\DeclareMathOperator{\PSL}{PSL}

\DeclareMathOperator{\Sym}{Sym}

\usepackage{color}

\begin{document}
\title[Commuting probability]{Coprime automorphisms of profinite groups and the commuting probability of  Sylow subgroups}

\author[E. Detomi]{Eloisa Detomi}
\address{Dipartimento di Matematica \lq\lq Tullio Levi-Civita\rq\rq, Universit\`a degli Studi di Padova, Via Trieste 63, 35121 Padova, Italy} 
\email{eloisa.detomi@unipd.it}
\author[R. M. Guralnick]{Robert M. Guralnick}
\address{ Department of Mathematics, University of
Southern California, Los Angeles, CA90089-2532, USA}
\email{guralnic@usc.edu}
\author[M. Morigi]{Marta Morigi}
\address{Dipartimento di Matematica, Universit\`a di Bologna\\
Piazza di Porta San Donato 5 \\ 40126 Bologna \\ Italy}
\email{marta.morigi@unibo.it}
\author[P. Shumyatsky]{Pavel Shumyatsky}
\address{Department of Mathematics, University of Brasilia\\
Brasilia-DF \\ 70910-900 Brazil}
\email{pavel@unb.br}

%\subjclass[2020]{20D20; 20E45; 20P05} 
% 20D20 Sylow subgroups, Sylow properties, π-groups, π-structure 
%20E45 conjugacy classes, 
%20P05 Probabilistic methods in group theory da [3]  20N99 None of the above, but in this section, in 20Nxx Other generalizations of groups, 
%\keywords{Commuting probability, Sylow subgroups, Simple groups}
% 20D20 Sylow subgroups, Sylow properties, π-groups, π-structure 
% 20D45 Automorphisms of abstract finite groups
%20P05 Probabilistic methods in group theory da [3]  20N99 None of the above, but in this section, in 20Nxx Other generalizations of groups, 

\begin{abstract} Given two (closed) subgroups $H,K$ of a profinite group $G$, we write $\pr(H,K)$ for the probability that a random pair from $H\times K$ commutes. Here we are concerned with profinite groups containing Sylow subgroups $P,Q$ such that $\pr(P,Q)>0$. First, we show that if $G$ is a profinite group containing a Sylow $2$-subgroup $L$ and a Sylow $p$-subgroup $P$, where $p$ is odd, such that $\pr(L,P)$ is positive, then $G$ is virtually pro-$p$-soluble (Theorem \ref{main1}). Then we handle similar issues for profinite groups admitting coprime automorphisms. In particular, we prove that if $G$ is a profinite group admitting a group of coprime automorphisms $A$ such that there is a Sylow $2$-subgroup $L$ of $C_G(A)$ and an $A$-invariant Sylow $p$-subgroup $P$ of $G$, where $p$ is odd, for which $\pr(L,P)>0$, then $G$ is virtually pro-$p$-soluble (Theorem \ref{main3}). On the other hand, if we only have $\pr([L,A],[P,A])>0$, then $[G,A]$ need not be virtually pro-$p$-soluble. We show that in this case $[G,A]$ has an open normal subgroup of non-$p$-soluble length at most $1$ (Theorem \ref{main44}). Furthermore, if $P$ is an $A$-invariant Sylow $p$-subgroup of $G$ such that $\pr([P,A],[P,A]^x)>0$ for every $x\in G$, then $[G,A]$ has an open normal subgroup of non-$p$-soluble length at most $1$ (Theorem \ref{main5}).
\end{abstract}
 \maketitle

\section{Introduction}

Every compact group $G$ can be viewed as a probabilistic space using the normalized Haar measure. Given two (closed) subgroups $H,K\leq G$, we write $\pr(H,K)$ for the probability that a random pair from $H\times K$ commutes. Here we are concerned with profinite groups containing Sylow subgroups $P,Q$ such that $\pr(P,Q)>0$. The following theorem was established in \cite{MZ}.\medskip

\noindent{\bf Theorem A.} {\it Let $G$ be a profinite group containing a Sylow $2$-subgroup $L$, a Sylow $3$-subgroup $P$ and a Sylow $5$-subgroup $Q$ such that $\pr(L,P)$ and $\pr(L,Q)$ are both positive. Then $G$ is virtually prosoluble.}
\medskip

As usual, we say that a profinite group has a property virtually if it has an open subgroup with that property. Our first result in this paper extends Theorem A as follows.

\begin{theorem}\label{main1}
 Let $G$ be a profinite group containing a Sylow $2$-subgroup $L$ and a Sylow $p$-subgroup  $P$, where $p$ is an odd prime, 
  such that $\pr(L,P)$ is positive. Then $G$ is virtually pro-$p$-soluble.
\end{theorem}

Recall that a finite group is $p$-soluble if its composition factors are either $p$-groups or $p'$-groups. Thus, we say that a profinite group is pro-$p$-soluble if it is an inverse limit of finite $p$-soluble groups. Theorem \ref{main1} is indeed an extension of Theorem A, since any finite group that is both 3-soluble and 5-soluble is in fact soluble (this is straightforward from Thompson's classification of minimal insoluble groups \cite{thompson}).

We say that a (continuous) automorphism $\al$ of a profinite group $G$ is coprime if it has finite order and induces a coprime automorphism on every $\al$-invariant finite homomorphic image of $G$, that is, $(|G/N|,|\alpha|) = 1$ for every $\al$-invariant open normal subgroup $N$ of $G$.

Suppose a profinite group $G$ admits a finite group of coprime automorphisms $A$. 
If $H\leq G$ is an $A$-invariant subgroup, we sometimes will write $H_A$ to denote $H \cap C_G(A)$, 
 where $C_G(A)$ is the set of elements of $G$ fixed by all automorphisms in $A$. 
 It is well-known that any Sylow subgroup of $C_G(A)$ is of the form $P_A$, where $P$ is an $A$-invariant Sylow subgroup of $G$. 

Observe that even if $C_G(A)$ is soluble, the group $G$ need not be virtually prosoluble. Indeed, let $G=\prod_{i\in I} S_i$ be a Cartesian product of infinitely many groups $S_i$, all isomorphic to $\PSL(2,2^5)$. Let $\alpha$ be an automorphism of $G$ which acts on each $S_i$ as  a field automorphism  of order $5$. Then $\alpha$ is a coprime automorphism of $G$ whose centralizer in $G$ is soluble, being the Cartesian product of infinitely many groups isomorphic to $\PSL(2,2)$, but $G$ is not virtually prosoluble.

In view of Theorem A and the above example, the following theorem is of interest. 

\begin{theorem}\label{main2} Let $G$ be a profinite group admitting a finite group of coprime automorphisms $A$ and containing $A$-invariant Sylow $2$-subgroup $L$, Sylow $3$-subgroup $P$ and Sylow $5$-subgroup $Q$ such that $\pr(L_A,P_A)$ and $\pr(L_A,Q_A)$ are both positive. Then $G$ is virtually prosoluble.
\end{theorem}

Note that a there is no direct analogue of Theorem \ref{main2} related to Theorem \ref{main1}, that is, if a profinite group $G$ admits a coprime automorphism $\al$ and contains $\al$-invariant Sylow $2$-subgroup $L$ and Sylow $p$-subgroup $P$ such that $\pr(L_\al,P_\al)>0$, then $G$ need not be virtually pro-$p$-soluble. For example, let again $G=\prod_{i\in I} S_i$ be the Cartesian product of infinitely many groups $S_i$, all isomorphic to $\PSL(2,2^5)$, and let $\alpha$ be an automorphism of $G$ which acts on each $S_i$ as  a field automorphism  of order $5$. 
Observe that $|G|$ is divisible by 11 while $C_G(\al)$ is a pro-$\{2,3\}$ group. Choose $\al$-invariant Sylow 2-subgroup $L$ and Sylow 11-subgroup $P$. Note that $P_\al=1$ and therefore $\pr(L_\al,P_\al)=1$ while $G$ is not pro-11-soluble.

As a counterweight to the above example, we have the following theorem.

\begin{theorem}\label{main3}
Let $G$ be a profinite group admitting a finite group of coprime automorphisms $A$. Assume that $G$ has $A$-invariant Sylow $2$-subgroup $L$ and Sylow $p$-subgroup $P$, where $p$ is an odd prime,  such that $\pr(L_A,P)>0$. Then $G$ is virtually pro-$p$-soluble.  
\end{theorem} 

Note that Theorem \ref{main1} is a special case of  Theorem \ref{main3}, where $A=1$. 

Recall that if a group $G$ admits a group of automorphisms $A$, the subgroup generated by the elements $x^{-1}x^\al$, where $x$ runs over $G$ and $\al$ over $A$, is denoted by $[G,A]$. This is an $A$-invariant normal subgroup of $G$. 

Following \cite{KhSh15}, we say that a finite group $G$ has non-$p$-soluble length $\lambda_p(G)$ at most $k$ if there is a normal series all of whose quotients are either $p$-soluble or direct products of nonabelian simple groups of order divisible by $p$ with at most $k$ non-$p$-soluble quotients. In particular, $\lambda_p(G)=0$ if and only if $G$ is $p$-soluble. We say that a profinite group $G$ has non-$p$-soluble length $\lambda_p(G)$ at most $k$ if $G$ is an inverse limit of finite groups with that property. It follows that 
such a group possesses a normal series all of whose quotients are either pro-$p$-soluble or Cartesian products of nonabelian finite simple groups of order divisible by $p,$ with at most $k$ non-pro-$p$-soluble quotients (see Wilson's paper \cite{Wilson}).

Suppose a profinite group $G$ admits a group of coprime automorphisms $A$ such that there is an $A$-invariant Sylow $2$-subgroup $L$ and  an $A$-invariant Sylow $p$-subgroup $P$, where $p$ is an odd prime, for which $\pr([L,A],[P,A])>0$. In view of our previous experience -- in particular, Theorem \ref{main1} -- it would be natural to expect that $[G,A]$ is virtually pro-$p$-soluble. However, this turns out to be false and $[G,A]$ need not be virtually pro-$p$-soluble (see  Example \ref{es1}).  Therefore the following theorem is in a sense best possible.

\begin{theorem}\label{main44}
Let $G$ be a profinite group admitting a finite group of coprime automorphisms $A$ such that there is an $A$-invariant Sylow $2$-subgroup $L$ and an $A$-invariant Sylow $p$-subgroup $P$, where $p$ is an odd prime, for which $\pr([L,A],[P,A])>0$. Then $[G,A]$ has an open normal subgroup of non-$p$-soluble length at most $1$. 
 \end{theorem} 

Our next result is related to the following theorem obtained in \cite{DGMS1}.\medskip

\noindent{\bf Theorem B.} {\it Let $G$ be a profinite group containing a Sylow $p$-subgroup $P$ such that $\pr(P,P^x)>0$ for all $x\in G$. Then $O_{p,p'}(G)$ is open in $G$.}
\medskip

We address the question on the structure of a profinite group $G$ admitting a finite group of coprime automorphisms $A$ and containing an $A$-invariant Sylow $p$-subgroup $P$ such that $\pr([P,A],[P,A]^x)>0$ for every $x\in G$. It turns out that even if $G=[G,A]$, the group $G$ need not be virtually pro-$p$-soluble and in particular $O_{p,p'}(G)$ need not be open in $G$ (see Example \ref{remark}). Instead, the conclusion about $G$ is similar to that in Theorem \ref{main44}.

\begin{theorem}\label{main5}
 Let $G$ be a profinite group admitting a finite group of coprime automorphisms $A$. Assume that $P$ is an $A$-invariant Sylow $p$-subgroup of $G$ such that $\pr([P,A],[P,A]^x)>0$ for every $x\in G$. 
Then $[G,A]$ has an open normal subgroup of non-$p$-soluble length at most $1$. 
 \end{theorem} 

The proofs of the above results depend on the classification of finite simple groups. Throughout, by a finite simple group we mean a finite nonabelian simple group.

\section{Preliminaries} 

We first present some general results about commuting probabilities in compact groups. So, unless otherwise specified, $G$ will be a compact Hausdorff topological group and a subgroup of $G$ will always be a closed subgroup. Note that any finite index subgroup of $G$ is open in $G$. We recall  that a profinite group  is a  compact  Hausdorff topological space, where a basis of  neighbourhoods of the identity consists of normal subgroups of finite index. 
 
The Borel $\sigma$-algebra $\mathcal M$ of a compact group $G$ is the one generated by all closed subsets of $G$. We say that a measure $\mu$ on $(G,\mathcal M)$ is a (left) Haar measure provided $\mu$ is both inner and outer regular, $\mu(K) < \infty$ and $\mu(xE) = \mu(E)$ for all compact subsets $K$ and measurable subsets $E$ of $G$ (see \cite[Chapter 4]{HR} or \cite[Chapter II]{Nach}). Recall that there is a unique Haar measure $\mu$ on $(G,\mathcal M)$ such that $\mu(G) = 1$. 

Moreover, if $H$ is a subgroup of $G$, then either $\mu(H) =0$ or $\mu(H) >0$, and in the latter case $H$ is open in $G$ and $\mu(H)=|G:H|^{-1}.$

If $H$ and $K$ are  subgroups of $G$ the
set 
\[C = \{(x, y) \in H \times K \mid xy = yx\}\]  is closed in $H \times K$ since it is the  preimage of $1$ under the continuous map $f : H\times K \rightarrow G$ given by $f(x, y) = [x, y]$.
Denoting the normalized Haar measures of $H$ and $K$ by $\mu_H$ and $\mu_K$, respectively, the probability that a random element from $H$ commutes with a random element from $K$ is defined as 
 \[\pr(H,K) = (\mu_H \times \mu_K)(C).\]

%Note that 
%\[\pr(H,K) = \int_H \mu_K (C_K(x)) d\mu_H(x)  = \int_K \mu_H (C_H(y)) d\mu_K(y).\]
%
%
%For every $x \in G$, the centralizer $C_G(x)$ equals $f_x^{-1} (1)$, where $f_x$ is the continuous function $f_x(y) = [x, y]$. It follows that $C_G(x)$ is closed and measurable. 
%We define 
%\[ \pr (x, K)= \mu_K \left(C_K(x)\right), \]
%and note that $\pr (x, K)= |K:C_K(x)|^{-1}$ whenever $C_K(x)$ is open in $K$. 
	
The following lemmas can be found for instance in \cite{MZ}, (see Lemma 5 and Lemma 6) and they show that the commuting probability of two subgroups does not decrease when passing to homomorphic images or considering proper subgroups.

\begin{lemma}\label{sub} 
Let $H,K$ be subgroups of a compact group $G$. Then for any subgroup $H_0 \le H$ we have 
\[\pr(H_0,K)\geq \pr(H,K).\]
\end{lemma}
\begin{lemma}\label{quot} 
 Let $G$ be a compact group and let $N$ be a normal subgroup of  $G$.  For any subgroups $H,K\leq G$ we have 
\[\pr(H,K)\leq \pr(HN/N,KN/N)\pr(N\cap H,N\cap K).\]
\end{lemma}

In particular, the previous lemmas show that if a profinite group $G$ has a Sylow $p$-subgroup $P$ and a Sylow $2$-subgroup $L$ such that $\pr(L,P)\ge \eta$, then every normal  section of $G$ has the same property. We will often use this fact without mentioning. 

The following is Lemma 7.1 in \cite{DGMS1}. 

\begin{lemma}\label{limit} 
Let $G$ be a profinite group and $H, K \le G$.  Then 
\[ \pr(H,K) = \inf_{N \lhd_o G} \pr \left(\frac{HN}{N},\frac{KN}{N}\right).\]
\end{lemma}\

\begin{lemma}\label{prod} 
Let $G$ be a profinite group and $H, K \le G$. Assume that $H=H_1H_2$ is the product of two subgroups. Then
\[\pr(H,K)\ge\pr(H_1,K)\pr(H_2,K)\]
\end{lemma}
\begin{proof} First assume that $G$ is finite. We will repeatedly use tha fact that, if $A$ and $B$ are subgroup of $G$ then every element of the product $AB$ % can be writte in exactly $|H_1\cap H_2|$ ways as a product $h=h_1h_2$, with $h_1\in H_1,h_2\in H_2$. In particular, $|H|= (|H_1|\,|H_2|)/|H_1\cap H_2|$. 
can be written in exactly $|A\cap B|$ ways as a product $ab$, with $a\in A$, $b\in B$. In particular, $|AB|= (|A|\,|B|)/|A\cap B|$. 
We have that
\begin{eqnarray*}
\pr(H,K)
&=&
\frac{1}{|H||K|}
\sum_{h\in H} |C_K(h)|
 = %&=&
\frac{|H_1\cap H_2|}{|H_1||H_2||K|}\sum_{h_1h_2\in H} |C_K(h_1h_2)|\\
&=&
\frac{1}{|H_1||H_2||K|}\sum_{h_1\in H_1}\sum_{h_2\in H_2}
|C_K(h_1h_2)|.
\end{eqnarray*}
%\geq \frac{1}{|H_1||H_2||K|}\sum_{h_1\in H_1}\sum_{h_2\in H_2}|C_K(h_1)\cap C_K(h_2)|=\frac{1}{|H_1||H_2||K|}\sum_{h_1\in H_1}\sum_{h_2\in H_2}
%\frac{|C_K(h_1)|\,|C_K(h_2)|}{|C_K(h_1)C_K(h_2)|}\ge \frac{1}{|H_1||H_2||K|}\sum_{h_1\in H_1}\sum_{h_2\in H_2}\frac{|C_K(h_1)|\,|C_K(h_2)|}{|K|}
 Note that 
 \begin{eqnarray*}
|C_K(h_1h_2)| &\ge& |C_K(h_1)\cap C_K(h_2)|=\frac{|C_K(h_1)|\,|C_K(h_2)|}{|C_K(h_1)C_K(h_2)|} \\
&\ge& \frac{|C_K(h_1)|\,|C_K(h_2)|}{|K|}.
\end{eqnarray*}

Therefore 
 \begin{eqnarray*}
\pr(H,K) &\ge &\sum_{h_1\in H_1}\sum_{h_2\in H_2}\frac{|C_K(h_1)|\,|C_K(h_2)|}{|H_1||H_2||K|^2}\\
 &=& \left(\sum_{h_1\in H_1}\frac{|C_K(h_1)|}{|H_1||K|}\right)\,\left(\sum_{h_2\in H_2}\frac{|C_K(h_2)|}{|H_2||K|}\right) \\
 &=& \pr(H_1,K)\pr(H_2,K).
\end{eqnarray*}

If $G$ is a profinite group, the result follows from Lemma \ref{limit}, taking into account the fact that, by Lemma \ref{quot}, the commuting 
probability of two subgroups does not decrease when considering homomorphic images of $G$.
\end{proof}

In what follows we use the following well-known facts (see for example \cite{asch}), often without mention. 
\begin{lemma}\label{coprime}
Let a group $A$ act coprimely on a finite group $G$. The following  holds:
\begin{itemize}
    \item[(i)] $G = [G, A]C_G(A)$ and $[G,A]=[G,A,A]$;\
		\item[(ii)] if $G$ is abelian, then $G = [G,A]\times C_G(A)$;
    \item[(iii)] if $N$ is any $A$-invariant normal subgroup of $G$, we have $C_{G/N}(A) = C_G(A)N/N$;
		\item[(iv)] the group $G$ possesses an $A$-invariant Sylow $p$-subgroup for each prime $p \in \pi(|G|)$,  any two $A$-invariant Sylow $p$-subgroups are conjugate by an element of $C_G(A)$, and any  $A$-invariant  $p$-subgroup of $G$ is contained in an $A$-invariant Sylow $p$-subgroup;
		\item[(v)] if $N$ is any normal subgroup of $G$ such that $[N,A]=1$, then $[G,A]$ centralizes $N$.
    \end{itemize}
\end{lemma}

\section{The case of finite simple groups}
 We will often use without special references the well-known corollary of the classification that if a simple group $G$ admits a group of coprime automorphisms $A$ of order $e\ne1$, then $G = L(q)$ is a group of Lie type, $A=\langle\al\rangle$ is cyclic, and $\alpha$ is a field automorphism. Furthermore, $C_G(\alpha) = L(q_0)$ is a group of the same Lie type defined over a smaller field such that $q = q_0^e$ (see \cite{GLS3}).

We will need the following result, which is Theorem 2.2 in \cite{BFMMNSST}.

\begin{theorem} Let $S$ be a finite nonabelian simple group of order divisible by an
odd prime $p$. Then $S$ contains a conjugacy class of $p$-elements of even size.\end{theorem}

A straightforward consequence of the above result is the following corollary, which will be fundamental in the sequel.

\begin{corollary}\label{cor-simple-p-q}  Let $S$ be a finite simple group and let $p$ be an odd prime dividing the order of $S$. Then no Sylow $2$-subgroup of $S$ is centralized by a Sylow $p$-subgroup.
\end{corollary}

\begin{lemma}\label{cent-small}  Let $G$ be a finite simple group admitting a nontrivial 
 coprime automorphism $\alpha$. Let $r$ be an odd prime divisor of $|C_G(\alpha)|$. Then no Sylow $r$-subgroup of $C_G(\alpha)$ centralizes a Sylow $2$-subgroup of $C_G(\alpha)$. Moreover, either $3$ or $5$ divides the order of $C_G(\alpha)$.
\end{lemma}
\begin{proof} We have that $G = L(q)$ is a group of Lie type, $\alpha$ is a field automorphism and $C_G(\alpha) = L(q_0)$
is again a finite simple group unless it is one of the following groups:   $A_1(2)$, $A_1(3)$, $B_2(2)$, $G_2(2)$, $^2A_2(2^2)$, $^2B_2(2)$, $^2G_2(3)$ and $^2F_4(2)$. In the above list $A_1(2)$, $A_1(3)$, $B_2(2)$, $^2A_2(2^2)$ and $^2B_2(2)$ are soluble of order divisible by $2$ or $5$, and it can be checked directly that in those cases no  Sylow $r$-subgroup of $C_G(\alpha)$  centralizes a Sylow $2$-subgroup of $C_G(\alpha)$. In the other cases, we have that $C_G(\alpha)'$ is simple and $r$ divides $C_G(\alpha)'$. So for these groups, as well as for the cases where $C_G(\alpha)$ is a simple group, Corollary \ref{cor-simple-p-q} applies. Since the order of a finite simple group is always divisible by $3$ or $5$, the result follows.
\end{proof}

%{\color{blue} $A_1(2)$ (isomorphic to $\Sym(3)$), $A_1(3)$ (isomorphic to the alternating group $\alt(4)$), $B_2(2)$ (which is isomorphic to $\Sym(6)$), $G_2(2)$ (whose derived subgroup is simple), $^2A_2(2^2) \cong PSU_3(2)$ (solvable of order 72) $^2B_2(2)$  (solvable of order 20), $^2G_2(3)$  (whose derived subgroup is simple) and $^2F_4(2)$ (whose derived subgroup is simple).}

\begin{lemma}\label{lemma-simple-coprime} Let $G$ be  a  finite simple group of Lie type in characteristic $\elp$ admitting a coprime automorphism  $\alpha$. Let $r\ne\elp$ be a prime divisor of $|C_G(\al)|$.
Then any $\alpha$-invariant Sylow $r$-subgroup of $G$ is contained in $C_G(\alpha)$. 
\end{lemma}

\begin{proof} 
If $\al$ is trivial, then the result is obviously true. Otherwise,  $\alpha$ is a field automorphism of $G = L(q)$ and $C_G(\alpha) = L(q_0)$, where $q=q_0^e$ and $e=|\al|$. 
 By hypothesis,  $q_0$ as well as $q$ are $\elp$-powers. 
 
 We look at the formulae for the orders of finite simple groups (see  e.g. \cite{Carter2}) and we see that, apart from the case of $^3D_4(q)$,  the $\elp'$-divisors of  the index of $ L(q_0)$ in  $L(q)$ are divisors of a product of the form
 \begin{equation}\label{prodS}
  \prod_{i \in I} \frac{ (q_0^e)^{d_i} \pm 1 }{ q_0^{d_i} \pm 1}
  \end{equation}
  where each $d_i$ is a positive integer. 
  
  Note that  $\pi(d_i) \subseteq \pi (|L(q)|)$. 
   Indeed, if $p \neq t \in \pi(d_i)$, then, from  the formula of the order of $L(q)$, we see that  $(t-1)$ divides $d_j$ for some $j$;
     then, since $t$ divides $(q^{t-1}-1)$, it follows that $t$ divides $|L(q)|$. 

 When $G=^3\!\!D_4(q)$ the same argument applies, noting that $q^8+q^4+1=(q^{12}-1)/(q^4-1)$.

Keeping in mind that $\al$ is a coprime automorphism, we deduce that $(d_i, e)=1$ for every $i$.  Since $r$ divides $|L(q_0)|$,  
 the order of   $q_0$ in the field with $r$ elements divides $d_i$ (or $2d_i)$ for some $i$, and so it is prime to $e$. 
   This means that when $r$ divides the factor $\left((q_0^e)^j \pm 1\right)$,  then it divides $(q_0^j \pm 1)$ too.  
   Moreover, if  $r$ divides $(q_0^j \pm 1)$, then 
  $r$  does not divide $\left((q_0^e)^j \pm 1\right)/(q_0^j \pm 1)$. 
  This shows that $r$ does not divide the product \eqref{prodS}
 and so  the index of $L(q_0)$ in $L(q)$ is an $r'$-number.

So $C_G(\alpha)$ contains an $\alpha$-invariant Sylow $r$-subgroup of $G$, and then also 
 any other  $\alpha$-invariant Sylow $r$-subgroup of $G$, since they are all conjugate by  elements of $C_G(\alpha)$. 
 \end{proof}

\begin{theorem}\label{thm-simple-coprime} Let $G$ be a finite simple group and let  $\alpha$ be   a coprime automorphism of $G$.
 Let $r$ be an odd prime divisor of $|G|$. Then no Sylow $r$-subgroup of $G$ centralizes a Sylow $2$-subgroup of $C_G(\alpha)$. 
\end{theorem}
\begin{proof} 
If $\al$ is trivial, then the result is Corollary \ref{cor-simple-p-q}. 

Otherwise, $G = L(q)$ is a group of Lie type, $\alpha$ is a field automorphism and $C_G(\alpha) = L(q_0)$.  
Assume $q_0$ even. 
Note that  $C_G(\alpha) =L(q_0)$ contains a regular unipotent element $x$ (see e.g. Proposition 5.1.7 of \cite{Carter2}).
 It follows from Corollary 4.6 of \cite{Hum} that $C_G(x)$ is contained in a   Sylow $2$-subgroup of $G$. 
% \comment{check, we don't have the file }
 Thus no Sylow $r$-subgroup of $G$  centralizes a   Sylow $2$-subgroup of $G$. 
  
 So we are left with the case where $q_0$ is odd. By Lemma \ref{lemma-simple-coprime}, 
  a Sylow $2$-subgroup of $C_G(\al)=L(q_0)$ is also a Sylow $2$-subgroup of $G$. Therefore the result follows from  Corollary \ref{cor-simple-p-q}. 
\end{proof}

\section{Commuting probability of subgroups of coprime orders}\label{tres}  

Somewhat abusing terminology, we say that a Sylow $p$-subgroup of a group $G$ is trivial if $p$ does not divide the order of $G$.

\begin{lemma}\label{product_prob}
  Let $G=S_1\times\dots\times S_t$ be the direct product of finite simple groups ($t$ factors). 
Assume that $A$ is a finite group of coprime automorphisms of $G$. Let $\LL$ be an $A$-invariant Sylow $2$-subgroup, $P$ an $A$-invariant Sylow $3$-subgroup, and $Q$ an $A$-invariant Sylow $5$-subgroup  of $G$. 
 Then: 
\begin{enumerate}
\item  $\pr (\LL_A ,P_A) \pr(\LL_A ,Q_A) \le (3/4)^{t/|A|} $;
\item  If $r$ is an odd prime dividing $|S_i|$ for every $i$ and $R$ is an $A$-invariant Sylow $r$-subgroup  of $G$,  then $\pr (\LL_A,R)  \le (3/4)^{t/|A|}$.
\end{enumerate}
 \end{lemma}

\begin{proof} 
 
\noindent (1) Set $a=|A|$. Note that $A$ acts on $\{S_1,\dots, S_t\}$ by permuting the simple factors $S_1,\dots,S_t$ and each orbit has length at most $a$.
If $n$ is the number of orbits, then $t\le an$. 

Renumbering, if needed, the simple factors, assume that 
\[U=S_1\times\dots\times S_k\]
 is the product of simple factors in one $A$-orbit. Let $B=N_A(S_1)$ be the stabilizer of $S_1$ under the permutational action.
 Then every element of $C_U(A)$ is of the form $(x,x^{\alpha_1},\dots,x^{\alpha_{k-1}})$ for some $x\in C_{S_1}(B)$ and $\alpha_i\in A$. Thus, there is a natural isomorphism between $C_U(A)$ and $C_{S_1}(B)$.

Observe that $B$ induces a cyclic (possibly trivial) group of coprime automorphisms of $S_1$. Let $\beta$ be a generator of that cyclic group, so that 
 $C_{S_1}(B)=C_{S_1}(\beta)$. 
   Let $\LL_\beta,P_\beta$ and $ Q_\beta$ be the Sylow subgroups of $C_{S_1}(\beta)$ corresponding to $\LL_A,P_A$ and $ Q_A$, respectively, under the above isomorphism. 

  By Lemma \ref{cent-small} and Corollary \ref{cor-simple-p-q}, either $3$ or $5$ divides the order of $C_{S_1}(\beta)$ and either $[\LL_\beta,P_\beta]\ne 1$ or $[\LL_\beta,Q_\beta]\ne 1$ (and both could occur simultaneously). Therefore either $[\LL_A\cap U,P_A\cap U]\ne 1$ or $[\LL_A\cap U,Q_A\cap U]\ne 1$. It follows from Lemma 2.5 and Remark 2.6 in \cite{DLMS-finite} 
  % {\color{blue} (see also Lemma 2.4 in {DGMS1})}
   that  $\pr (\LL_A\cap U,P_A\cap U) \le 3/4$ or  $\pr (\LL_A\cap U,Q_A\cap U) \le 3/4$. Therefore 
\[ \pr (\LL_A\cap U,P_A\cap U)\pr (\LL_A\cap U,Q_A\cap U)  \le 3/4.  \] 

Then, considering the products of simple factors in all orbits $U_1,\dots,U_n$, we get
\begin{eqnarray*}
\pr (\LL_A,P_A) \pr(\LL_A,Q_A) &=& \prod_{i=1}^{r} \pr (\LL_A\cap  U_i,P_A\cap  U_i) \pr(\LL_A\cap  U_i,Q_A\cap  U_i) \\
&\le& (3/4)^r \le (3/4)^{t/a}.
\end{eqnarray*}

\noindent (2) We use the same notation as above. Note that 
 $R\cap S_1$  is a Sylow $r$-subgroup of $S_1$.  By Theorem  \ref{thm-simple-coprime} no Sylow $r$-subgroup of $S_1$ centralizes a Sylow $2$-subgroup of $C_{S_1}(\beta)$, 
 therefore $[\LL_A,R\cap S_1]\ne 1$. It follows that $[\LL_A\cap U,R\cap U]\ne 1$ and thus
\[ \pr (\LL_A\cap U,R\cap U ) \le 3/4. \]
Therefore, considering the products of simple factors in all orbits $U_1,\dots,U_n$, we get 
\[\pr (\LL_A,R) =\prod_{i=1}^{r} \pr (\LL_A\cap  U_i,R\cap  U_i) \le  (3/4)^r \le  (3/4)^{t/a}, \]
as claimed. 
\end {proof}

Every finite group $K$ possesses a series
\[1 = K_0 \le K_1 \le\dots \le K_{2h+1} = K\]
of normal subgroups such that $K_{i+1}/K_i$ is soluble (possibly trivial) if $i$ is
even and a direct product of nonabelian simple groups if $i$ is odd. Following \cite{KhSh15} the minimum number
of insoluble factors in a series of this kind is called the insoluble length $\lambda(K)$ of $K$.  Wilson showed in \cite{Wilson} that if $\mathcal{X}_1,\dots, \mathcal{X}_n$ are classes of finite groups closed with  respect to normal subgroups and subdirect products and 
if $\mathcal{X}$ is the class of finite groups having a normal series of given length $n$ such that
the $i$-th section is a $\mathcal{X}_i$-group, then any pro-$\mathcal{X}$ group has a normal series of length $n$  such that
the $i$-th section is a pro-$\mathcal{X}_i$-group. In particular, a combination of Lemma 2 and Lemma 3 of
\cite{Wilson} shows that if $\mathcal{X}$ is the class of finite groups $K$ such that $\lambda(K) \le l,$ then
any pro-$\mathcal{X}$ group has a normal series of  length at most $2l+1$ each of whose factors is
either prosoluble or a Cartesian product of nonabelian finite simple groups.

For any prime $p$ the non-$p$-soluble length $\lambda_p(G)$ of a finite group is defined by replacing ``soluble" by ``$p$-soluble". Note that, for any positive integer $k$, a finite group has a largest normal subgroup of non-$p$-soluble length at most $k$.
As above, if $\mathcal{X}$ is the class of finite groups $K$ such that $\lambda_p(K) \le l,$ then
any pro-$\mathcal{X}$ group $G$ has a normal series of  length at most $2l+1$ each of whose factors is
either pro-$p$-soluble or a Cartesian product of nonabelian finite simple groups of order divisible by $p$. 

Recall that the generalized Fitting subgroup $F^*(G)$ of a finite group
$G$ is the product of the Fitting subgroup $F(G)$ and all subnormal quasisimple subgroups. Here a group is quasisimple if it is perfect and its
quotient by the centre is a nonabelian simple group.

 The following lemma 
 uses the Schreier conjecture, i.e. the fact that  the outer automorphism groups of finite simple groups are soluble, which is a consequence of the classification of finite simple groups.

\begin{lemma}\label{lambda_total} 
Let $G$ be a finite group. Assume that 
 \[F^{*}(G)=S_1\times\dots\times S_k\]
 is a direct product of $k$ finite simple groups. 
 Then the soluble radical of $G/F^{*}(G)$ has index at most $k!$. 
\end{lemma} 
\begin{proof}  
Note that $C_G(F^*(G)) \le Z(F^*(G))$ (see, for instance, \cite[Theorem 9.8]{Isaacs}). Therefore $G$ is isomorphic to a subgroup of the automorphism group $\aut(F^*(G))$, which in turn can be embedded in the semidirect product $B\rtimes \Sym(k)$, where the symmetric group $\Sym(k)$ acts on $B=\aut(S_1)\times\dots\times\aut(S_k)$ by permuting the factors (here the action is not necessarily transitive). 

Therefore $\bar G=G/F^*(G)$ is isomorphic to a subgroup of $\bar B\rtimes\Sym(k),$ where $\bar B=\out(S_1)\times\dots\times\out(S_k)$. For each $i$, the outer automorphism group $\out(S_i)$ is soluble by the Schreier conjecture. Since $\bar G/\bar B$ is isomorphic to a subgroup of $\Sym(k)$, it follows that  the soluble radical of $\bar G$ has index at most $k!$, as desired. 
\end{proof}

We are now ready to prove Theorem \ref{main3}.

 Note that Theorem \ref{main1} is a special case of Theorem \ref{main3}, with $A$ trivial.
 
 We will make use of the well-known fact that if a profinite group $G$ is finite-by-pro-$p$-soluble, then it is virtually pro-$p$-soluble. Indeed, if $K$ is a finite normal subgroup of $G$ such that $G/K$ is pro-$p$-soluble, then $C_G(K)$ is an open pro-$p$-soluble subgroup. 
 
Moreover, if $G$ admits a finite group of automorphisms $A$, then for every open normal subgroup $N$ the subgroup 
 $\cap_{\al\in A} N^{\al}$ is open in $G$. Therefore, we can always assume that  $G$ is an inverse limit of finite groups $G/N$, where $N$ is $A$-invariant, so that $A$ induces a group of  automorphisms on $G/N$. 
  Furthermore, if $A$ is coprime,  then $C_G(A)N/N=C_{G/N}(A)$ by Lemma \ref{coprime}.

\begin{proof}[Proof of Theorem \ref{main3}] 
Recall that  $G$ is a profinite group admitting a  group of coprime automorphisms $A$ of finite order $a=|A|$. Moreover there exists  $\eta>0$, a Sylow $2$-subgroup $ \LL_A$ of $C_G(A)$ and a Sylow $p$-subgroup $P$ of $G$ such that $\pr(P,  \LL_A) \ge \eta$. We want to prove that  $G$ is virtually pro-$p$-soluble.

  Note that, if $N$ is an $A$-invariant normal subgroup of $G$, the hypotheses are inherited both by $N$ and $G/N$. 
  
Let   $\bar G=G/N$  be a finite continuous image  of $G$, where $N$ is a normal $A$-invariant subgroup of $G$, and use the bar notation for the subgroups of $\bar G$. 
 We claim that $\lambda_p(\bar G)$ is $(\eta, a)$-bounded. Without loss of generality, we can assume that the $p$-soluble radical of $\bar G$ is trivial. Then 
 \[ \bar K=F^{*}(\bar G)=S_1\times\dots\times S_t \]
  is a direct product of finite simple groups of order divisible by $p$. 
 Note that there is a Sylow $2$-subgroup $\tilde \LL_A$ of $C_{\bar K}(A)$ and a Sylow $p$-subgroup $\tilde P$ of $\bar K$ such that $\pr ( \tilde P, \tilde \LL_A)\ge\eta$. 
 By Lemma \ref{product_prob}(b), 
\[ \pr(\tilde P, \tilde \LL_A) \le  (3/4)^{t/a}.\] 
It follows that $ (3/4)^{t/a} \ge  \eta$ and so $t \le a \log_{3/4} \eta$ is  $(\eta, a)$-bounded. 
 Therefore,  by Lemma \ref{lambda_total}, 
 we conclude that $\lambda_p(\bar G)$
  is   $(\eta, a)$-bounded, as claimed. 

 As $N$ was an arbitrary open normal $A$-invariant subgroup of $G$ and such subgroups form a basis for neighbourhoods of the identity in $G$, it follows from \cite{Wilson} that  $G$ has a normal $\alpha$-invariant series of $(\eta,a)$-bounded length each of whose factors is either pro-$p$-soluble or a Cartesian product of  finite simple groups whose order is divisible by $p$. 

Let $T$ be a factor in this series of the form  $T=\prod_{i\in I} S_i$, where each $S_i$ is a  finite simple group and $p$ divides $|S_i|$. 
The above argument shows that $|I| \le a \log_{3/4} \eta$ is $(\eta,a)$-bounded and, in particular, $T$ is finite. 

Now the conclusion follows by induction on $\lambda_p(G)$. Namely, if $\lambda_p(G)=0$ then $G$ is pro-$p$-soluble and the result holds. Assume that 
$\lambda_p(G)=h\ge 1$ and let $R$ be the maximal normal pro-$p$-soluble subgroup of $G$. Passing to the quotient $G/R$, without loss of generality we can assume that $R=1$. In view of the above $G$ possesses a finite normal subgroup $T$, which is a direct product of  simple groups whose order is divisible by $p$, such that $\lambda_p(G/T)=h-1$. By induction $G/T$ is virtually pro-$p$-soluble. It follows that $G$ is virtually pro-$p$-soluble as well. Now the proof is complete.
\end{proof}

The proof of Theorem \ref{main2} can be obtained using similar arguments. 

\begin{proof}[Proof of Theorem \ref{main2}] 
Recall that  $G$ is a profinite group admitting a  group of coprime automorphisms $A$ of finite order $a=|A|$. Moreover  $G$  contains $A$-invariant Sylow $2$-subgroup $\LL$,  Sylow $3$-subgroup $P$ and Sylow $5$-subgroup $\R$ such that $\pr( \LL_A,  P_A) \ge \eta$ and $\pr( \LL_A,  \R_A) \ge \eta,$ for some $\eta>0$.
  We want to prove that  $G$ is virtually prosoluble. 

Note that if $N$ is an $A$-invariant normal subgroup of $G$, the hypotheses are inherited by both $N$ and $\bar G=G/N$.  

So, let $\bar G$ be such a quotient with $N$ open in $G$, and let
 \[ \bar K=S_1\times\dots\times S_t\]
 be an $A$-invariant normal subgroup of  $\bar G$
  which is a direct product of finite simple groups. Observe that $C_{\bar K}(A)$ contains a Sylow $2$-subgroup $\tilde \LL_A$, a Sylow $3$-subgroup 
    $\tilde P_A$ and a Sylow $5$-subgroup $\tilde \R_A$ 
    such that $\pr(\tilde \LL_A, \tilde P_A) \ge \eta$ and $\pr(\tilde \LL_A, \tilde \R_A) \ge \eta$. 
By  Lemma \ref{product_prob}(a), 
\[ \pr(\tilde \LL_A, \tilde P_A) \pr(\tilde \LL_A, \tilde \R_A) \le  (3/4)^{t/a},\]
whence $ (3/4)^{t/a} \ge  \eta^2$ and so $t \le a \log_{3/4} \eta^2$  is  $(\eta, a)$-bounded. 
 
We deduce from Lemma \ref{lambda_total} that $\lambda(G/N)$ is   $(\eta, a)$-bounded.  As $N$ was an arbitrary open normal $A$-invariant subgroup of $G$ and such subgroups form a basis for neighbourhoods of the identity in $G$, it follows from \cite{Wilson} that $G$ has a normal $A$-invariant   series of $(\eta,a)$-bounded length each of whose factors is either prosoluble or a Cartesian product of nonabelian finite simple groups. By the above, if an $A$-invariant section $\bar K$ is isomorphic to a Cartesian product of nonabelian finite simple groups, then it is finite. Since a finite-by-prosoluble group is virtually prosoluble, by induction on the length of the series, we conclude that $G$ is virtually prosoluble. 
\end{proof}

\section{Theorem \ref{main44}}  %%%%%%%%%%%%%%%%%%%%%%%%%%%%%%%%%%%%%%%%% 

 We start this section with an example showing that if a profinite group $G$ satisfies the assumptions of Theorem \ref{main44}, 
 then $G$ need not be virtually pro-$p$-soluble.
 
 \begin{example}\label{es1}
{\rm Let $S$ be a finite simple group admitting a coprime automorphism $\al$ and assume that $S$ is of Lie type in an odd characteristic. 
  Then any $\alpha$-invariant Sylow $2$-subgroup of $S$ is contained in $C_S(\alpha)$ (see Lemma \ref{lemma-simple-coprime}). 
 Let $G$ be an infinite Cartesian product of groups isomorphic to $S$. Then $G$ admits the coprime automorphism that acts on every simple factor of $G$ as $\al$. We denote the automorphism of $G$ by the same symbol $\al$ and choose an $\al$-invariant Sylow 2-subgroup $L$ and an $\al$-invariant Sylow $p$-subgroup $P$, where $p$ is an odd prime. Note that $[L,\al]=1$ and therefore $\pr([L,\al],[P,\al])=1,$ while the prosoluble radical of $G$ is trivial.} $\qed$
\end{example}

We now proceed with the proof of Theorem \ref{main44}, which requires some technical lemmas first.

 \begin{lemma}\label{transitive} 
 Let $G$ be a finite group admitting a finite group of coprime automorphisms $A$  such that there is an $A$-invariant Sylow $2$-subgroup 
$L$ and an $A$-invariant Sylow $p$-subgroup $P$, where $p$ is an odd prime, for which $\pr([L,A],[P,A]) \ge \ep >0$. 
   Assume that $G=S_1\times\dots\times S_t$ is the direct product of finite simple groups ($t$ factors) where $p$ divides $|S_i|$ for every $i$, and $A $ transitively permutes the factors $S_1, \dots , S_t$. Then $t$ is $\ep$-bounded. 
\end{lemma} 
\begin{proof} 
We can clearly assume $t > 1$.  
We have that $L=L_1\times\dots\times L_t$ and  $P=P_1\times\dots\times P_t$, where 
 $L_i$ and $P_i$ are nontrivial Sylow subgroups of  $S_i$ for each $i$. 

Choose $\al \in A$ such that $S_1^\al \neq S_1$. 
If $x \in S_1$ and $\pi_1: G \rightarrow S_1$ is the projection on $S_1$, then $\pi_1 ([x,\al])=x^{-1}$. Therefore
$\pi_1 ([L,\al])=L_1$ and $\pi_1 ([P,\al])=P_1$.  
It follows from Corollary \ref{cor-simple-p-q} that $[L_1,P_1] \neq 1$. Hence $\pr (L_1, P_1) \le 2/3$. 

Similarly, $ \pr(L_i, P_i) \le  2/3$ for every $i$ so that
\begin{equation}\label{0}
 \pr(L,P) = \prod_i  \pr (L_i, P_i) \le (2/3)^t .  
\end{equation}

On the other hand, as $\pi_1 ([L,A])=L_1$ and $\pi_1 ([P,A])=P_1$, 
\[ \pr (L_1, P_1) \ge \pr( [L,A],[P,A])\ge \ep. \] 
Moreover, since $L_1$ centralizes $P_j$ for $j\neq 1$, we have 
\begin{equation}\label{1}
 \pr(L_1 ,P) =   \pr (L_1, P_1)\ge \ep.  
\end{equation}
In particular, 
$ \pr(L_1 ,[P, A]) \ge    \pr (L_1, P)\ge \ep .$ 
Since $L= L_1 ^A= L_1 [L,A]$, it follows from the hypothesis and  Lemma \ref{prod} that 
\begin{eqnarray*}  \pr(L ,[P, A]) &=& \pr(L_1 [L,A] ,[P, A])\\
 &\ge& \pr(L_1  ,[P, A]) \pr( [L,A] ,[P, A]) \ge \ep^2.\end{eqnarray*}
Similarly, $   \pr(P ,[L, A])\ge \ep^2$. 
Then,  by Lemma \ref{prod} and \eqref{1}, 
\[  \pr(L , P) \ge  \pr(L_1 ,P) \pr([L, A], P) \ge \ep^3.\] 
Comparing with \eqref{0}, we conclude that $(2/3)^t \ge  \pr(L , P)  \ge \ep^3$, so that $t$ is $\ep$-bounded. 
\end{proof}

 \begin{lemma}\label{non-transitive} 
 Let $G$ be a finite group admitting a group of coprime automorphisms $A$ such that there is an $A$-invariant Sylow $2$-subgroup $L$ and an $A$-invariant Sylow $p$-subgroup $P$, where $p$ is an odd prime, for which $\pr([L,A],[P,A])\ge\ep>0$. Assume that $G=S_1\times\dots\times S_t$ is the direct product of simple groups of order divisible by $p$. Then there exists an $\epsilon$-bounded integer $m$ such that at most $m$ of the factors $S_i$ are not $A$-invariant. 
\end{lemma} 
\begin{proof}
Write 
\[ G=T_1\times\dots\times T_s,\] 
where $T_i$ are minimal normal $A$-invariant subgroups of $G$. 
 We have that
  $L=L_1\times\dots\times L_s$ and  $P=P_1\times\dots\times P_s$, where 
 $L_i$ and $P_i$ are Sylow subgroups of  $T_i$.   
 In particular 
 \[ \pr([L,A],[P,A]) = \prod_{i=1}^s \pr([L_i,A],[P_i,A]), \]
  and $ \pr([L_i,A],[P_i,A]) \ge \pr([L,A],[P,A]) \ge \ep$. 
 Thus  we can apply Lemma \ref{transitive} to $T_i$  and deduce that $T_i$ is a product of $\ep$-boundedly many simple factors $S_j$ of $G$.   
 
If $T_i$ is simple, then it is an $A$-invariant simple factor $S_j$, for some $j$. So we can assume that each $T_i$ is not simple and we need to prove that $s$ is $\ep$-bounded. 

 As in Lemma \ref{transitive}, since $T_i = S_{i_1} \times \cdots \times S_{i_t}$ is a product of simple groups, transitively permuted by $A$, and $t>1$, 
  the projection $\bar L_i$ of $[L_i, A]$ on the first component $S_{i_1}$  is a Sylow $2$-subgroup of $S_{i_1}$. Similarly, 
    the projection $\bar P_i$ of $[P_i, A]$ on the first component $S_{i_1}$  is a (non-trivial) Sylow $p$-subgroup of $S_{i_1}$.  
  Therefore, by Corollary \ref{cor-simple-p-q},  $[\bar L_i, \bar P_i] \neq 1,$ hence 
  \[   \pr([L_i,A],[P_i,A])  \le  \pr (\bar L_i, \bar P_i) \le 2/3.\]  
Since this holds for every $i$, we get 
\[  \pr([L,A],[P,A]) = \prod_{i=1}^s  \pr([L_i,A],[P_i,A]) \le (2/3)^s.\] 
 As $ \ep \le  \pr([L,A],[P,A])$ by hypothesis,  $s$ is $\ep$-bounded and the proof is complete. 
\end{proof}

We need the following lemma, which is  \cite[Lemma 2.5]{AGS1}. 
\begin{lemma}\label{2.5}
Let $M = H\langle a \rangle$ be a group with a normal subgroup $H$ and an element 
$a$ such that $(|H|, |a|) = 1$ and $H = [H, a]$. Suppose that $M$ faithfully acts 
by permutations on a set $\Omega$ in such a way that the element $a$ moves only $m$ 
points. Then the order of $M$ is $m$-bounded.
\end{lemma}

If $G$ is a finite group, let $R_p(G)$ be the $p$-soluble radical of $G$. Then $U(G)/R_p(G)=F^*(G/R_p(G))$ is a direct product of  simple groups of order divisible by $p$. Let $K_p(G)/R_p(G)$ be the kernel of the permutational action of $G$ on the simple factors of $U(G)/R_p(G)$. Since $K_p(G)/U(G)$ is isomorphic to a subgroup of a direct product of  outer automorphism groups $\out(S_i)$ of  simple groups $S_i$, and  $\out(S_i)$ is soluble by the Schreier conjecture, we deduce 
 that $K_p(G)/U(G)$ is soluble. Therefore 
\[\lambda_p(K_p(G))\le 1. \]

\begin{proof}[Proof of Theorem \ref{main44}] Recall that $G$ is a  profinite group admitting a finite group of coprime automorphisms $A$ such that there is an $A$-invariant Sylow $2$-subgroup $L$ and an $A$-invariant Sylow $p$-subgroup $P$, where $p$ is an odd prime, for which $\pr([L,A],[P,A])=\ep>0$. 
We want to prove that $[G,A]$ has an open normal subgroup of non-$p$-soluble length at most $1$. 

It is sufficient to prove that there exists an $\epsilon$-bounded integer $n$ such that, for every open normal $A$-invariant subgroup $N$   of $G,$ the subgroup $K_p(G/N)$ has index at most $n$ in $G/N$. 

Thus, we now assume that $G$ is finite. Furthermore, without loss of generality, we assume that $R_p(G)=1$ and $F^*(G)$ is a direct product of simple groups 
    \[F^*(G)=S_1\times\dots\times S_t\] of order divisible by $p$.

Let $K=K_p(G)$, and let $B$ be the kernel of the permutational action of $A$ on the set $\{S_1,\dots,S_t \}$. Set $\bar A=A/B$ and $\bar G=G/K$. We will use the bar notation in $\bar A$ and $\bar G$. Now it is sufficient to prove that $[\bar G,\bar A]$ has $\epsilon$-bounded order. 
  
By Lemma \ref{non-transitive} applied to $F^*(G)$, there is an $\epsilon$-bounded integer $m$ such that the number of factors  in the set $\Omega=\{ S_1, \dots , S_t \}$ moved by $\bar A$  is at most $m$. In particular the order of $\bar A$ is at most $m!$.  

Let $\al \in A$; as $\al$ moves at most $m$ factors, it follows from Lemma \ref{2.5} applied to $[\bar G,\bar \al]\langle \bar \al \rangle$ that the there exists an integer $n,$ depending only on $m$, such that order of $[\bar G,\bar \al]\langle \bar\al \rangle$ is at most $n$. Note that $n$ is $\epsilon$-bounded.
 
Therefore
\[ \bar G=[\bar G,\bar A]= \prod_{\bar\al \in \bar A} [\bar G,\bar\al],\] 
 being a product of at most $m!$ normal subgroups of order at most $n$, has $\epsilon$-bounded order. As $G/K_p(G)$ is isomorphic to $\bar G$, this concludes the proof.
\end{proof}

\section{Theorem \ref{main5}}\label{last} 

We will use the well-known fact that if $G$ is a profinite group,  $H$ a  subgroup of $G$, and $H_0$ is an open subgroup of $H$, then there exists an open normal subgroup $M$ of $G$ such that $M\cap H\le H_0.$ Moreover, if $G$ admits a finite group of automorphisms $A$, we may assume that $M$ is $A$-invariant, by replacing it with $\cap_{\alpha\in A}M^{\alpha}$, which is open as well because $A$ is finite. 

\begin{lemma}\label{open} Let $G$ be a profinite group admitting a coprime automorphism $\alpha$. 
Let $P$ be an $\alpha$-invariant Sylow $p$-subgroup of $G$ and $H$  an open $\al$-invariant normal subgroup of $G$ such that $[P\cap H,\al] $ is finite, then  $[P,\alpha]$ is finite. 
\end{lemma} 
\begin{proof}
 Let $M$ be an $\alpha$-invariant open normal subgroup of $G$ such that $[(P\cap H),\alpha]\cap M=1$. By replacing $M$ with $M\cap H$, we may assume that $M\le H$, so that $[(P\cap M),\alpha]=1$.  Now $[P,\alpha]\cap M$ is a normal subgroup of $[P,\alpha]$ centralizing $\alpha,$ so by Lemma \ref{coprime} (v) it is central in $[P,\alpha].$
Therefore the center of $[P,\alpha]$ has finite index in $[P,\alpha]$ and it follows from  Schur's Theorem \cite[4.12]{Rob} that $[P,\alpha]'$ is finite.

Considering the natural action of $\alpha$ on the abelian group $[P,\alpha]/[P,\alpha]'$, we deduce from Lemma \ref{coprime} (ii) and (iii) that $C_{[P,\alpha]}(\alpha)\le [P,\alpha]'$, which is finite. 

Hence, $[P,\alpha]\cap M=C_{[P,\alpha]\cap M}(\alpha)$ is finite. It follows that  $[P,\alpha]$  is finite as well. This completes the proof.
\end{proof}

\begin{lemma}\label{KpP} Let $G$ be a profinite group admitting a coprime automorphism $\alpha$. Let $P$ be an $\alpha$-invariant Sylow $p$-subgroup of $G$ such that $[P,\alpha]$ is abelian. Let $M$ be an $\alpha$-invariant normal subgroup of $G$ which is a Cartesian product of simple groups, each of order divisible by $p$. Then $[P,\alpha]$ normalizes every simple factor of $M$.
\end{lemma}

\begin{proof} The group $G\langle\alpha\rangle$ acts on $M$ by permuting the simple factors. 
As $G$ is not soluble and $\al$ is coprime, in view of the Odd Order Theorem \cite{fetho}  the order of $\al$ is odd.

Suppose that $S_1$ is a simple factor of $M$ that is not $\alpha$-invariant. Then we may assume that $S_1^\alpha=S_2$ for some factor $S_2$ of $M$, with $S_2\ne S_1$. Note that if $S_1^{\alpha^{-1}}=S_2$ then $\alpha^2$ normalizes $S_1$ while $\alpha$ doesn't, contradicting the fact that the order of $\alpha$ is odd. Therefore we may assume that $S_1^{\alpha^{-1}}=S_3$.  If $x\in P\cap S_1$, then $[x,\alpha]=x^{-1}x^\alpha\in S_1\times S_2$. As $[P,\alpha]$ is abelian, it centralizes $x^{-1}x^\alpha$, therefore it normalizes $S_1\times S_2$. If $[P,\alpha]$ does not normalize $S_1$, then $\{S_1,S_2\}$ is an orbit under the action of $[P,\alpha]$. As $[P,\alpha]=[P,\alpha^{-1}]$, the same argument shows that $\{S_1,S_3\}$ is an orbit under the action of $[P,\alpha]$, contradicting the fact that $S_2\ne S_3$. This shows that $[P,\alpha]$ normalizes $S_1$ whenever $S_1$ is not $\alpha$-invariant. 

Now assume that $S_1^{[P,\alpha]}=\prod_{j\in J}S_j$, where each $S_j$ is  $\alpha$-invariant. For each $j\in J$, let $P_j=S_j\cap P$ and let $P_0=\prod_{j\in J}P_j$. If $[P_j, \alpha]\ne 1$, then $[P,\alpha]$ centralizes a nontrivial element of the form  $x^{-1}x^\alpha,$ with $x\in P_j$ and it normalizes $S_j$. Therefore we may assume that $[P_j,\alpha]=1$ for every $i\in J$.  
Then $\alpha$ centralizes $P_0$. Note that $P_0$ is a normal subgroup of $P_0[P,\alpha]$ centralized by $\al$. 
By Lemma \ref{coprime} (v), we deduce that  $[P,\alpha]$ centralizes $P_0$. 
In particular  $[P,\alpha]$ centralizes $P_1$ and therefore it normalizes $S_1$, so that $|J|=1$. This concludes the proof. 
\end{proof}

\begin{corollary}\label{Kp} Let $G$ be a finite group admitting a coprime automorphism $\alpha$. Let $P$ be an $\alpha$-invariant Sylow $p$-subgroup of $G$ such that $[P,\alpha]$ is abelian. Then $[P,\alpha]\le K_p(G)$.
\end{corollary}
\begin{proof} We can assume that $R_p(G)=1$ and 
\[F^*(G)=S_1\times\cdots\times S_t\] 
where each $S_i$ is a 
  simple group of order divisible by $p$. The group $G$ acts on $F^*(G)$ by permuting the factors $S_1,\dots, S_t$ and $K_p(G)$ is the kernel of the permutational action. The result follows from Lemma \ref{KpP}.
\end{proof}

\begin{lemma}\label{Oppp} Let $G$ be a finite $p$-soluble group admitting a coprime automorphism $\al$ and having an $\al$-invariant Sylow $p$-subgroup $P$  such that $[P,\alpha]$ is abelian. Then $[P,\alpha]\le O_{p',p,p',p,p',p}(G)$.
\end{lemma}

\begin{proof} The subgroup $[P,\alpha]$ is a normal abelian subgroup of $P$. It was proved by Hall and Higman in  Theorem 3.2.1 of \cite{HH} that any abelian normal sugroup $A$ of a Sylow $p$-subgroup of a finite $p$-soluble group $G$ is contained in $O_{p',p}(G)$ if $p\ge 5$. Hartley's Theorem 2 in \cite{Har} deals with the situation where $p\in\{2,3\}$, stating that in those cases $A\le O_{p',p,p',p,p',p}(G)$. The lemma follows.
\end{proof}

\begin{lemma}\label{Op_finite} Let $G$ be a profinite group such that $O_p(G)=1$ and let $N$ be a finite normal subgroup of $G$. Then $O_p(G/N)$ is finite. \end{lemma}

\begin{proof} Let $T/N=O_p(G/N)$ and let $P$ be a Sylow $p$-subgroup of $T$. Note that $T=PN$. We have $C_P(N)\le O_p(T)\le O_p(G)=1,$ therefore $P$ acts faithfully on $N$, which is finite. It follows that $P$ is finite, and so is $T$.
\end{proof}

The following proposition is of independent interest. It provides a crucial step in the proof of Theorem \ref{main5}.

\begin{proposition}\label{main4}
Let $G$ be a profinite group admitting a finite group of coprime automorphisms $A$. Assume that $P$ is an $A$-invariant Sylow $p$-subgroup of $G$ such that $\pr([P,A],[P,A]^x)>0$ for every $x\in G$. Then $[P,A]O_p(G)/O_p(G)$ is finite.
 \end{proposition} 

\begin{proof}  
Fix an automorphism $\al \in A$. We first prove that the subgroup $[P,\al]O_p(G)/O_p(G)$ is finite. 

Since $ \pr([P,\alpha], [P,\alpha])=\epsilon$, for some positive $\epsilon$, a profinite version of  Neumann's theorem (see \cite{neumann} and \cite{LP}) shows that  $[P,\alpha]$ is  virtually abelian. So there exists an $\alpha$-invariant open normal subgroup $H$ of $G$ such that $[P,\alpha]\cap H$ is abelian. 
 By Lemma \ref{open}, 
 it enough to prove that the image of $[(P\cap H),\alpha]$ in $G/O_p(G)$ is finite, and thus, replacing $G$ with $H$, 
we may assume that $[P,\alpha]$ is abelian. 

Let $\bar G=G/N$ be a finite continuous image of $G$, where the subgroup $N$ is $\alpha$-invariant. As above, let $\bar R_p$ be the $p$-soluble radical of $\bar G$ and let $K_p(\bar G)$ be the kernel of the permutational action of $\bar G$ on the simple factors of $F^*(G/R_p(G))$. It follows from Lemma \ref{Kp} that $[\bar P,\alpha]\le K_p(\bar G)$, where $\lambda_p(K_p(\bar G))\le 1$.
 If  $T$ is the normal closure of $[P,\alpha]$ in $G$, then $\bar T\le  K_p(\bar G)$ and so $\lambda_p(T)\le1.$
As $[P,\alpha]\le T$, it follows that $[P\cap T,\alpha]= [P,\alpha],$ so we can replace $G$ with $T$, and we may assume that $G$ has a characteristic series
\[1\le R\le M\le G\]
where $R=R_p(G)$ is the pro-$p$-soluble radical, $M/R$ is a Cartesian product of simple groups of order divisible by $p$ and $G/M$ is pro-$p$-soluble. 

 Firstly, we want to show that $[P,\alpha]$ is finite modulo $R_p(G)$. 
 We  can assume that $R_p(G)=1$ and $M=\prod_{i\in I} S_i$, where $S_i$ are
 simple groups of order divisible by $p$. As $C_G(M)\cap M=1$, we have that $C_G(M)$ is isomorphic to its image in $G/M$, which is pro-$p$-soluble. Hence, $C_G(M)=1$. This implies that $C_{[P,\alpha]}(M)=1$ and, in particular, $O_p(M[P,\alpha])=1$. Since we want to prove that $[P,\al]$ is finite, we can work with the subgroup $M[P,\al]$. It follows from Lemma \ref{KpP} that  $[P,\al]$ normalizes each simple factor of $M$. 

Since 
\[ \pr([P,\alpha], [P,\alpha]^g) >0, \ \textrm{ for every } \ g \in M, \]
the subgroup $M$ is the union of  the sets
  \[ X_n=\{g \in M \mid \pr([P,\alpha], [P,\alpha]^g) \ge 1/n\}. \]
Note that the set $X_n$ is closed for every $n \in \N$. Indeed, 
 consider an element $a \notin X_n$; then $\pr([P,\alpha], [P,\alpha]^a) < 1/n$. It follows from Lemma \ref{limit}
 that there is an open normal subgroup $N$ of $G$ such that $ \pr ([P,\alpha]N/N, [P,\alpha]^{a}N/N)  < 1/n$. Observe that for all $x \in N \cap M$, 
   \[ \pr([P,\alpha], [P,\alpha]^{ax}) \le  \pr ([P,\alpha]N/N, [P,\alpha]^{a}N/N) < 1/n, \] 
   which shows that $ax \notin X_n$. Therefore, $M \setminus X_n$ is open and so $X_n$ is closed.

By the Baire category theorem (see \cite[p. 200]{Ke}) at least one of the sets $X_n$ has  non-empty interior. Hence, there is an open $\al$-invariant 
 normal subgroup $M_0$ of $M$ and an element $a\in M$ such that 
 \[ \pr([P,\alpha], [P,\alpha]^{ax})\ge 1/n\]
  for every $x \in M_0$. 

Since $M_0$ is open in $M$ and $M$ is a Cartesian product of simple groups, we can write $M=M_0 \times M_1$, where $M_1$ is a finite $\al$-invariant  normal subgroup of $M$. Note that $M_1$ is $[P,\alpha]$-invariant. 
 Passing to the quotient over $M_1$, we can assume that $M=M_0$  and $O_p(M[P,\alpha])$ is finite  by Lemma \ref{Op_finite}.  

Therefore  
\[ \pr([P,\alpha], [P,\alpha]^{x})\ge 1/n \]
 for every $x \in M$.
 By a profinite version of Theorem 1.3 of \cite{DGMS2}, we conclude that $[P,\al]$ is finite modulo $O_p(M [P,\al])$. Since $O_p(M[P,\alpha])$ is finite, $[P,\al]$ is finite as well.  
  
Thus, we have shown that $[P,\al]$ is finite modulo $R_p(G)$. 
  
Let $H$ be an open normal $\alpha$-invariant subgroup of $G$ with the property that $H \cap [P,\al] \le R_p(G)$. By Lemma \ref{open}, we can assume $G=H$, so that $[P, \al ] \le  R_p(G)$. 
 As  $[P\cap R_p(G),\alpha]= [P,\alpha],$   we can replace $G$ with $R_p(G)$  and so now $G$ is  pro-$p$-soluble.

It follows from Lemma \ref{Oppp} and a routine inverse limit argument that $[P,\alpha]\le O_{p',p,p',p,p',p}(G)$, so we can replace $G$ with $O_{p',p,p',p,p',p}(G)$ and assume that $G$ has finite $p$-length.

  It follows that $G$ possesses a normal series of finite length
  \[1 = L_0 \le L_1 \le\dots \le L_h=G,\]
 where every section is either a pro-$p$ or a pro-$p'$ group. Without loss of generality we may assume that $O_p(G)=1$. Now it is enough to prove that $[P,\al]$ is finite.

Let $U=O_{p'}(G)$ and observe that  $O_p(U [P,\al])=1$. 
Indeed, the subgroup $C=C_G(U)$ is normal in $G$ and $O_p(C)$ is trivial. If $C\ne1$, then $O_{p'}(C)\ne1$ and $O_{p'}(C)\leq Z(C)$. Therefore $O_{p',p}(C)$ is the direct product of its Sylow $p$-subgroup and  $O_{p'}(C)$. As $O_p(C)=1$, it follows that $O_{p',p}(C)=O_{p'}(C)$.  Since $C$ has a series with  pro-$p$ or a pro-$p'$ sections, we conclude that $C$ is a pro-$p'$ subgroup. In particular $[P,\al] \cap C=1$ and so $O_p(U [P,\al])=1$,  as claimed. 

Therefore we can assume that $G=U[P,\al]$ and $O_p(G)=1$. 

Now $[P,\al]$ is a Sylow $p$-subgroup of $G$. Hence, it follows from Theorem B in the introduction that $O_{p,p'}(G)$ is open in $G$. As $O_p(G)=1$,  we conclude that $[P,\al]$ is finite. 

Thus, we now know that $[P,\al]$ is finite modulo $O_p(G)$, for every $\al \in A$. 
 Since $A$ is finite, $[P,A]$ is a product of finitely many  subgroups $[P,\al]$, with $\al \in A$, which are all normal in $[P,A]$. Therefore, $[P,A]$ is finite modulo $O_p(G)$ as well. 
 \end{proof}

Now we prove Theorem \ref{main5}. 

\begin{proof}[Proof of Theorem \ref{main5}] Recall that $G$ is a profinite group admitting a finite group of coprime automorphisms $A$ and $P$ is an $A$-invariant Sylow $p$-subgroup of $G$ such that $\pr([P,A],[P,A]^x)>0$ for every $x\in G$. 
We want to prove that $[G,A]$ has an open normal subgroup of non-$p$-soluble length at most $1$. 

Without loss of generality assume that $G=[G,A]$.  By Proposition \ref{main4},  $[P,A]$ has finite order modulo $O_p(G)$. 
So, there exists an open $A$-invariant normal  subgroup $H$ of $G$ such that $H\cap [P,A]\le O_p(G)$. 

We will show that  $\lambda_p([H,A])\le 1$ by proving that, for every  open $A$-invariant normal subgroup $N$ of $G$  we have  $\lambda_p([HN/N,A])\le 1$. Let $\bar G=G/N$ and use the bar notation in the quotient. 

We may assume that $R_p(\bar H)=1$.
 So, 
\[ [\bar P\cap \bar H, A]\le \bar H\cap [\bar P,A]\le O_p(\bar G) \cap H\le R_p(\bar H)=1,\]
 that is, $A$ centralizes a Sylow $p$-subgroup of $\bar H$. Let $F^*(\bar H)=\prod_{i\in I} \bar S_i$, where each $\bar S_i$ is a simple group of order divisible by $p$. Let $\tilde P$ be an $A$-invariant Sylow $p$-subgroup of $F^*(\bar H)$ and note that $\tilde P \le \bar P\cap \bar H$ is centralized by $A$. 
 By the Frattini argument, we have that $\bar H=F^*(\bar H)\bar K$, where $\bar K=N_{\bar H}(\tilde P)$ is the normalizer in $\bar H$ of $\tilde P$. Moreover, $[\bar H,A]=[\bar K,A]F^*(\bar H)$.

Now $\tilde P$ is a normal subgroup of $\bar K$ centralized by $A$. Therefore, by Lemma \ref{coprime} (v), $[\bar K,A]$ centralizes $\tilde P$. In particular, $[\bar K,A]$ centralizes $\tilde P\cap\bar S_i$ for every $i$.  It follows that $[\bar K,A]$ normalizes all simple factors $\bar S_i$ and therefore $[\bar K,A]\le K_p(\bar H)$. Thus 
\[ [\bar H,A]\le [\bar K,A]F^*(\bar H)\le K_p(\bar H),\]
 which has $p$-soluble length at most $1$. This proves the claim that  $\lambda_p([H,A])\le 1$.

 Let $L=[H,A]^G$. Note that the  subgroups $[H,A]^g$, with $g\in G$, are  normal in $H$ and they all have non-$p$-soluble length at most $1$. This implies that $\lambda_p( L )\le 1$. By Lemma \ref{coprime} (v),   the image of $H$ in $G/L$ is central in $G/L=[G,A]/L$.  It follows that the centre of $G/L$ has finite index in $G/L$ and so, by Schur's Theorem,  
  $G'L/L$ is finite. 
 Let $M$ be an open normal subgroup of $G$ such that $M\cap G'\le L$. Then $M/L$ is abelian,  so $\lambda_p(M)=\lambda_p(L)=1$. This concludes the proof.
 \end{proof}

We conclude this section with an example showing that if a profinite group $G=[G,A]$ satisfies the assumptions of Theorem \ref{main5}, then $G$ need not be virtually pro-$p$-soluble and in particular $O_{p,p'}(G)$ need not be open in $G$. 
 
 \begin{example}\label{remark}
{\rm Let $S$ be a finite simple group admitting a coprime automorphism $\al$ and assume that $S$ is of Lie type in an odd characteristic. 
   Then any $\alpha$-invariant Sylow $2$-subgroup of $S$ is contained in $C_S(\alpha)$ (see Lemma \ref{lemma-simple-coprime}). 
% Let $L$ be an $\al$-invariant Sylow 2-subgroup of $S$ and observe that $L=L_\al$ (see Lemma \ref{lemma-simple-coprime}). 
  Let $G$ be an infinite Cartesian product of groups isomorphic to $S$. Then $G$ admits the coprime automorphism that acts on every simple factor of $G$ as $\al$. We denote the automorphism of $G$ by the same symbol $\al$ and choose an $\al$-invariant Sylow 2-subgroup $P$. Note that $[P,\al]=1$ and therefore $\pr([P,\al],[P,\al]^x)=1$ for every $x\in G,$ while the prosoluble radical of $G$ is trivial.}
\end{example}

\section{Acknowledgments}
Most of this work was done during a visit by P.~Shumyatsky to the Department of Mathematics of the University of Padova. He thanks the department for excellent hospitality. 

E.~Detomi and M.~Morigi are members of GNSAGA (INDAM) and were funded by  Project 2022PSTWLB (subject area: PE - Physical Sciences and
Engineering) ``Group Theory and Applications". 

R.~M.~Guralnick was partially supported by a Simons Foundation Fellowship 00019819.

P.~Shumyatsky acknowledges financial support from CNPq and \lq\lq National Group for Algebraic and Geometric Structures, and their Applications\rq\rq (GNSAGA - INDAM).


\begin{thebibliography}{10}

\bibitem{AGS1}  C.~Acciarri, R.~M.~Guralnick, P.~Shumyatsky, Coprime automorphisms of finite
groups, Trans. Amer. Math. Soc. 375 (2022), 4549--4565.

\bibitem{asch} M.~Aschbacher,  Finite group theory.  
Cambridge Studies in Advanced Mathematics, 10. Cambridge University Press, Cambridge, 2000.

\bibitem{BFMMNSST} A.~Beltr\'an, M.~J.~Felipe, G.~Malle,
A.~Moret\'o, G.~Navarro, L.~Sanus, R.~Solomon, P.~H.~Tiep, Nilpotent
and abelian Hall subgroups in finite groups. Trans. Am. Math. Soc.
368 (2016), 2497–2513.
\bibitem{Carter2}  R.~W.~Carter, Finite groups of Lie type. Conjugacy classes and complex characters. Pure and Applied Mathematics (New York). A Wiley-Interscience Publication. John Wiley \& Sons, Inc., New York, 1985.

\bibitem{DGMS1} E.~Detomi, R.~M.~Guralnick, M.~Morigi, P.~Shumyatsky,   Commuting probability for conjugate subgroups of a finite group,  	arXiv:2505.10521

 \bibitem{DGMS2} E.~Detomi, R.~M.~Guralnick, M.~Morigi, P.~Shumyatsky,      Finite groups, commuting probability, and coprime automorphisms,  	arXiv:2511.07597 

\bibitem{DLMS-finite}  E.~Detomi, A.~Lucchini, M.~Morigi, P.~Shumyatsky,  Commuting probability for the Sylow subgroups of a finite group, 	 Isr. J. Math. (2025). https://doi.org/10.1007/s11856-025-2847-6
\bibitem{MZ} E.~Detomi, M.~Morigi, P.~Shumyatsky,  Commuting probability for the Sylow subgroups of a profinite group. Math. Z. 309 (2025), no. 3, Paper No. 52, 13 pp. 

\bibitem{fetho} W.~Feit, J.~G.~Thompson, Solvability of groups of odd order, Pacific Journal of Mathematics, 13 (1963), 775--1029.
 \bibitem{GLS3} D.~Gorenstein, R.~Lyons, R.~Solomon, The classification of the finite simple groups. Number 3. Part I. Chapter A. Almost simple $K$-groups. Mathematical Surveys and Monographs, 40.3. American Mathematical Society, Providence, RI, 1998.

\bibitem{HH}  P.~Hall, G.~Higman, On the $p$-length of $p$-soluble groups and reduction
theorems for Burnside's problem, Proc. London Math. Soc. 6 (1956) 1--40.

\bibitem{Har} B.~Hartley, 
Some theorems of Hall-Higman type for small primes.
Proc. London Math. Soc.  41 (1980), 340--362.

\bibitem{HR} E.~Hewitt, K.~Ross, Abstract Harmonic Analysis. Vol. I, Die Grundlehren der mathematischen Wissenschaften, Vol. 115, Springer, Berlin-G\"{o}ttingen-Heidelberg, 1963. 

\bibitem{Hum}
J.~E.~Humphreys,  Conjugacy classes in semisimple algebraic groups. Mathematical Surveys and Monographs, 43. American Mathematical Society, Providence, RI, 1995.

\bibitem{Isaacs}  I.\,M. Isaacs, Finite Group Theory, Graduate Studies in Mathematics, 92.
American Mathematical Society, Providence, RI, 2008.
 
\bibitem{Ke} J.~L.~Kelley. General topology. Toronto - New York - London: Van Nostrand, 1955.

 \bibitem{KhSh15} E.~Khukhro,  P.~Shumyatsky,  {Nonsoluble and non-p-soluble length of finite groups}, Israel J. Math. {207(2)} (2015), 507--525.

\bibitem{LP} L.~L\'evai, L.~Pyber, {Profinite groups with many commuting pairs or involutions}, Arch. Math. 75 (2000), 1--7.

\bibitem{Nach} L.~Nachbin,  The Haar Integral, Van Nostrand, Princeton, N.J-Toronto-London, 1965.
\bibitem{neumann} P.~M.~Neumann, Two Combinatorial Problems in Group Theory, Bull.
 London Math. Soc. { 21}  (1989), 456--458.

\bibitem{Rob} D.~J.~S.~Robinson, Finiteness conditions and generalized soluble groups. Part1. Springer-Verlag, New York-Berlin, 1972.
\bibitem{thompson} J.~G.~Thompson, Nonsolvable finite groups all of whose local subgroups are solvable,  Bull. Amer. Math. Soc., 74 (1968), 383--437.
\bibitem{Wilson}  J.~Wilson,  On the structure of compact torsion groups. Monatsh. Math. 96 (1983),
57--66.
\end{thebibliography}
\end{document}